\newif\ifsiaga

\siagafalse

\ifsiaga
  \documentclass[review]{siamonline171218}
  \usepackage{amssymb,amstext,amsfonts}
  \newsiamremark{example}{Example}
\else
  \documentclass[12pt]{extarticle}
  \usepackage{amsmath,amsthm,amssymb,amstext,amsfonts}
  \usepackage{hyperref}
  \usepackage[capitalise]{cleveref}
  \oddsidemargin \evensidemargin
  \newtheorem{example}{Example}
  \newtheorem{definition}{Definition}
  \newtheorem{lemma}{Lemma}
  \newtheorem{proposition}{Proposition}
\fi
\usepackage{xcolor}
\usepackage[textsize=tiny]{todonotes}

\newcommand{\FormatAuthor}[3]{
\begin{tabular}{c}
#1 \\ {\small\texttt{#2}} \\ {\small #3}
\end{tabular}
}

\ifsiaga
\title{On cycles of pairing-friendly elliptic curves\thanks{Submitted
    to the editors October 2, 2018.
\funding{Alessandro Chiesa was supported by the UC Berkeley Center for Long-Term Cybersecurity. Lynn Chua was supported by a UC Berkeley University Fellowship.}}}
\author{Alessandro Chiesa\thanks{Department of Electrical Engineering
    and Computer Science, UC Berkeley
    (\email{alexch@berkeley.edu}, \email{chualynn@berkeley.edu})}
\and Lynn Chua\footnotemark[2] \and Matthew Weidner\thanks{Computer Laboratory, Cambridge (\email{malw2@cam.ac.uk})}}
\headers{On cycles of pairing-friendly elliptic curves}{A.
  Chiesa, L. Chua, and M. Weidner}

\else
\title{On cycles of pairing-friendly elliptic curves}
\author{
\begin{tabular}[h!]{ccc}
\FormatAuthor{Alessandro Chiesa}{alexch@berkeley.edu}{UC Berkeley}
 & \FormatAuthor{Lynn Chua}{chualynn@berkeley.edu}{UC Berkeley}
 & \FormatAuthor{Matthew Weidner}{malw2@cam.ac.uk}{Cambridge} \\
&&
\end{tabular}
}
\fi

\begin{document}
\maketitle

\begin{abstract}
A cycle of elliptic curves is a list of elliptic curves over finite
fields such that the number of points on one curve is equal to the
size of the field of definition of the next, in a cyclic way. We study
cycles of elliptic curves in which every curve is
pairing-friendly. These have recently found notable applications in
pairing-based cryptography, for instance in improving the scalability
of distributed ledger technologies. 

We construct a new cycle of length 4 consisting of MNT curves, and characterize all the possibilities for cycles consisting of MNT curves. We rule out cycles of length 2 for particular choices of small embedding degrees. We show that long cycles cannot be constructed from families of curves with the same complex multiplication discriminant, and that cycles of composite order elliptic curves cannot exist. We show that there are no cycles consisting of curves from only the Freeman or Barreto--Naehrig families.
  
\ifsiaga
\else
\fi
\end{abstract}

\ifsiaga

\fi

\ifsiaga
\else
\clearpage
\tableofcontents
\clearpage
\fi

\section{Introduction}

A cycle of elliptic curves is a list of elliptic curves defined over finite fields in which the number of points on one curve equals the size of the field of definition of the next, cyclically.

\begin{definition}
\label{def:cycle}
An \emph{$m$-cycle of elliptic curves} is a list of $m$ distinct elliptic curves $E_1/\mathbb{F}_{q_1}$,$\ldots$, $E_m/\mathbb{F}_{q_m}$, where $q_1,\ldots, q_m$ are prime, such that the numbers of points on these curves satisfy 
\begin{align} \label{eqn:cycle-condition} 
\#E_1(\mathbb{F}_{q_1}) = q_2\,, \ldots,
\#E_i(\mathbb{F}_{q_i}) = q_{i+1}\,,\ldots,
\#E_m(\mathbb{F}_{q_m}) = q_1\,. \end{align}
\end{definition}

This notion was introduced in \cite{ss} with the name of \emph{aliquot cycles}. The case of 2-cycles of ordinary curves, also called \textit{amicable pairs}, was introduced in the context of primality proving by \cite{preda, dualelliptic} under the equivalent notion of \emph{dual elliptic primes} (see \cref{appendix:dualelliptic}).

Silverman and Stange \cite{ss} showed that cycles of arbitrary lengths exist, and gave conjectural estimates, for any elliptic curve $E/\mathbb{Q}$, of the number of prime pairs $(q_1, q_2)$ such that reducing $E$ modulo $q_1$ and $q_2$ gives an amicable pair. Cycles of elliptic curves were further studied in \cite{reciprocity, jones, parks, parks2}, and some of these works refined and proved on average the conjectured estimates, showing that amicable pairs are asymptotically common.

In \cite{snarks} the notion of cycles of elliptic curves was extended for applications to pairing-based cryptography.

\begin{definition}
\label{def:cycle-pf}
A \emph{pairing-friendly $m$-cycle of elliptic curves} is an $m$-cycle
such that every elliptic curve in the cycle is ordinary and has a small embedding degree.
\end{definition}

Pairing-friendly cycles were used in \cite{snarks} to achieve recursive composition of \emph{zkSNARKs} (also known as \emph{proof carrying data}). A zkSNARK is a cryptographic scheme that allows one party (the prover) to convince another party (the verifier) that the prover knows a certain secret, via a short proof that is cheap to verify and reveals no information about the secret. Efficient zkSNARK constructions are obtained via pairing-friendly elliptic curves, and the cycle condition in \cref{eqn:cycle-condition} enables their recursive composition, while avoiding expensive modular arithmetic across fields of different characteristics. (See \cite{snarks} for details.)

Practitioners are interested in recursive composition of zkSNARKs, because it can be used to boost the scalability of distributed ledger technologies \cite{Breitman17}. For example, there are commercial efforts in this space whose core technology \emph{is} recursive composition \cite{coda}, and such technology thus rests on properties of cycles of pairing-friendly elliptic curves.

This motivates the question: \emph{what types of pairing-friendly cycles exist?}

A pairing-friendly $2$-cycle can be obtained from pairing-friendly prime-order curves of embedding degrees $4$ and $6$ \cite{kt, snarks}. Beyond this, there are \emph{no} other known constructions, and very little is known about pairing-friendly cycles. Indeed, requiring a small embedding degree as in \cref{def:cycle-pf} is a strong restriction and techniques used in previous work to study aliquot cycles do not seem to apply to pairing-friendly cycles.

This is unfortunate because the aforementioned MNT cycle is not ideal for applications: its unequal embedding degrees make one curve less secure than the other and, moreover, the fact that both embedding degrees are so small implies that using the cycle at high security levels is inefficient. It would be desirable, e.g., to have a $2$-cycle with embedding degrees $(12,12)$ or $(20,20)$ and, more generally, to understand this mathematical object better.

\subsection{Overview of results}
\label{sec:overview-of-results}

The stark difference in the current understanding of pairing-friendly cycles when compared to aliquot cycles, as well as applications to pairing-friendly cryptography in the real world, motivates a systematic study of pairing-friendly cycles. In this paper we initiate such a study, and our main results are the following.
\begin{enumerate}

\item
Prior to this work, the \emph{only} construction of pairing-friendly cycles was a 2-cycle from a family of curves called \emph{MNT curves}, named after Miyaji, Nakabayashi, and Takano \cite{mnt}. A natural question to ask is: can one construct other cycles consisting of MNT curves? In this work, we construct a new pairing-friendly cycle of length 4 using MNT curves. We also characterize \emph{all} the possibilities for cycles consisting of MNT curves, showing that any MNT cycle must have length 2 or 4, and that the curves must have embedding degrees alternating between 4 and 6. See \cref{section:MNT} for details.

\item
We then study \emph{arbitrary} pairing-friendly $2$-cycles (not derived from a particular family). We prove that $2$-cycles of elliptic curves with embedding degrees $(5, 10)$, $(8, 8)$, or $(12, 12)$ do \emph{not} exist. The technique that we use relies on the fact that the cyclotomic polynomials of these embedding degrees have degree $4$. In particular, we do not know how to extend this result to any other embedding degrees $(k_1,k_2)$. See \cref{section:degree4} for details.

\item
We move to study pairing-friendly cycles of arbitrary length. One strategy to construct cycles could be to pick a parametrized family of elliptic curves and try to construct cycles consisting of curves from the same family (like for MNT curves). What must the parameters of the family satisfy for such constructions to be possible? We prove that if the curves have the same discriminant for complex multiplication $D>3$, then we cannot construct cycles of length greater than 2 (\cref{section:discriminant}). This implies that to construct elliptic curve cycles, we must use curves from families of varying discriminants.
 
\item
So far we discussed cycles consisting of elliptic curves of \emph{prime} order. What if we relax the definition of cycles to allow \emph{composite} (non-prime) order elliptic curves in which the number of points on one curve is a multiple of (but not necessarily equal to) the size of the field of definition of the next? We prove that composite-order cycles \emph{cannot exist} (see \cref{section:cofactor}). This is a strong restriction as it implies that we must construct cycles using pairing-friendly elliptic curves of \emph{prime} order. Unfortunately, there are very few constructions of families of such curves in the literature, \emph{regardless of cycles}.

\item Lastly, we study the other known families of pairing-friendly elliptic curves of prime order (apart from MNT curves): the Freeman curves and the Barreto--Naehrig curves. We prove that cycles within each of these families do not exist (\cref{section:other}). This means that, if one wants to obtain cycles from curve families, one must consider combinations of current families (or study future constructions of prime-order elliptic curves).

\end{enumerate}
Overall, cycles of pairing-friendly elliptic curves seem much harder to understand, and to construct, than cycles of arbitrary elliptic curves. While our results have for the most part established limitations of pairing-friendly cycles, our outlook is optimistic. Our work demonstrates that studying pairing-friendly cycles is tractable and, moreover, points the way to concrete research questions that could lead to more tools for studying these cycles. We thus conclude the introduction with a selection of open problems.

\subsection{Open problems}
\label{sec:open-problems}

\begin{enumerate}

\item Do there exist cycles consisting of elliptic curves with the \emph{same} embedding degree? The varying embedding degrees in current constructions of cycles is inconvenient because, in practice, curves in the cycle have different security levels.

\item Can we construct cycles of embedding degrees greater than 6? All known pairing-friendly cycles involve embedding degrees at most $6$, which means that it is inefficient to use such cycles at high security levels (e.g., 128 bits of security). It would be desirable to construct, or rule out, cycles of higher embedding degrees (say, 20).
  
\item In particular, can we construct 2-cycles of higher embedding degrees?  Our technique for ruling out pairs with embedding degrees $(5, 10)$, $(8, 8)$, or $(12, 12)$ sheds some light on other pairs $(k_1, k_2)$ for which $\Phi_{k_1}(x) = \Phi_{k_2}(-x)$, but it does not seem to extend to the case $\deg \Phi_{k_1}(x) > 4$. We believe that it would be especially interesting to study pairs with embedding degrees $(16, 16)$, which have cyclotomic polynomial $x^8 + 1$.

\item Do there exist cycles consisting of elliptic curves with the same discriminant and the same embedding degrees? Our work demonstrates that sharing the same discriminant is already quite limiting, and it would be interesting to understand how this requirement interacts with that of sharing the same embedding degree.

\item Are there cycles from combinations of MNT, Freeman, and Barreto--Naehrig curves? Our preliminary investigations via Gr\"obner bases suggest small cycles are unlikely, but the question remains open for arbitrary-length cycles.
  
\end{enumerate}

\section{Preliminaries}

\subsection{Elliptic curves and pairings}

Let $E$ be an elliptic curve over a finite field $\mathbb{F}_q$, where
$q$ is a prime. We denote this by $E/\mathbb{F}_q$, and we denote by
$E(\mathbb{F}_q)$ the group of points of $E$ over $\mathbb{F}_q$, with
order $n=\#E(\mathbb{F}_q)$. The \emph{trace} of $E/\mathbb{F}_q$ is
$t=q+1-n$. By Hasse's theorem \cite[Theorem V.1.1]{silverman}, $t$
satisfies $|t| \leq 2\sqrt{q}$. We say that $E$
is \emph{supersingular} if $t\equiv 0 \pmod{q}$, otherwise $E$ is \emph{ordinary}.

The \emph{endomorphism ring} $\mbox{End}(E)$ of $E$ consists of morphisms from $E$ to itself that are also group homomorphisms on its points. If $E$ is supersingular, then $\mbox{End}(E)$ is an order in a quaternion algebra. If $E$ is ordinary, then $\mbox{End}(E)$ is an order in an imaginary quadratic field $\mathbb{Q}(\sqrt{-D})$, for some positive squarefree integer $D$. We call $D$ the \emph{discriminant}, and we say that $E$ has \emph{complex multiplication} in $\mathbb{Q}(\sqrt{-D})$.\footnote{Some works use the convention that $D$ is negative. Throughout this work we take $D$ to be positive.}

Let $r\geq 2$ be an integer relatively prime to $q$. We denote the $r$-torsion points of $E$ by $E[r]$, and we denote the group of $r$-th roots of unity in the algebraic closure of $\mathbb{F}_q$ by $\mu_r$. The \emph{Weil pairing} is a bilinear non-degenerate map 
\begin{equation}
e_r \colon E[r] \times E[r] \rightarrow \mu_r\,.
\end{equation}

The \emph{embedding degree} with respect to $r$ is the smallest integer $k$ such that $r$ divides $q^k - 1$. In the case of prime-order curves, if $r=n$ we simply say that $E$ has embedding degree $k$.

The Weil pairing was first used in cryptography to reduce the discrete logarithm problem on $E[r]$ to a discrete logarithm problem in $\mu_r$, which is contained in $\mathbb{F}_{q^k}^*$ \cite{mov,fr}. Subsequently, starting with the work of \cite{BonehF03,Joux04}, the Weil pairing was used to achieve numerous cryptographic capabilities. For security, it is necessary to choose the embedding degree $k$ such that the discrete logarithm problem in $\mathbb{F}_{q^k}^*$ is computationally infeasible. On the other hand, the embedding degree cannot be too large, or the computation of the Weil pairing (which grows linearly in $k$) would not be efficient enough for cryptographic applications.

We say that an elliptic curve $E/\mathbb{F}_q$ is \emph{pairing-friendly} if $E(\mathbb{F}_q)$ has a large prime-order subgroup, and if the embedding degree is small (see \cite{taxonomy} for a more precise definition). A random elliptic curve has a large embedding degree and thus is not pairing-friendly. Constructing pairing-friendly curves with specified parameters is a difficult problem with strong practical motivations that has been extensively studied. It was shown in \cite{mov} that supersingular elliptic curves can have embedding degree at most $6$, and if the characteristic of $q$ is not $2$ or $3$, the embedding degree is at most $3$. As we are interested in large values of $q$ and higher values of $k$ for applications in cryptography, we focus on ordinary elliptic curves.

The known methods to construct ordinary pairing-friendly curves proceed by first finding parameters $q,r,t,k$ such that there exists an elliptic curve $E/\mathbb{F}_q$ with trace $t$, a prime-order subgroup of size $r$, and embedding degree $k$. The \emph{complex multiplication method} is then used to find the equation of the curve. This works if the \emph{CM equation} $4q-t^2 = Dy^2$ has a solution with $y\in\mathbb{Z}$ and small positive discriminant $D\in\mathbb{Z}$. Indeed, state-of-the-art algorithms run in time $O(D\,\mbox{polylog}\,D)$ and are only feasible for $D$ of size up to $10^{16}$ \cite{cm}.

It is useful to view the condition on the embedding degree via cyclotomic polynomials. Let $\Phi_m$ be the $m$-th cyclotomic polynomial (the minimal polynomial over the rationals of an irreducible $m$-th root of unity). It is known that (see for example \cite{washington})
\begin{equation}
\label{eq:cyclotomic}
x^m - 1 = \prod_{d|m} \Phi_d(x)\,.
\end{equation}

\begin{lemma}
\label{lem:k-phi}
Let $E/\mathbb{F}_q$ have prime order $n$. Then $E$ has embedding degree $k$ if and only if $k$ is minimal such that $n$ divides $\Phi_k(q)$.
\end{lemma}

\begin{proof}
The condition that $k$ is the embedding degree implies that $k$ is minimal such that $q^k \equiv 1\pmod{n}$. Using basic results on cyclotomic polynomials (see \cite[Lemma 2.9]{washington}), this is equivalent to the condition that $n|\Phi_k(q)$.
\end{proof}

This can be converted into a result relating $n$ to the trace $t$. 

\begin{lemma}[\cite{bls}]
\label{lem:phi-t-1}
$E/\mathbb{F}_q$ has embedding degree $k$ if and only if $n|\Phi_k(t-1)$ and $n\nmid \Phi_i(t-1)$ for all $0<i<k$.
\end{lemma}

\subsection{Families of pairing-friendly elliptic curves}
\label{section:families}

We consider families of pairing-friendly elliptic curves with a fixed embedding degree, whose parameters are defined by polynomials. These are useful for generating curves for applications, where curves of arbitrary size are desired. Each family is parametrized by polynomials $(q_k(x), n_k(x), t_k(x))$, representing the field of definition, number of rational points, and trace respectively, where $k$ is the embedding degree. These have to satisfy that $n_k(x)=q_k(x)+1-t_k(x)$, $n_k(x)$ divides $\Phi_k(t_k(x)-1)$, and there must be infinitely many integer solutions $(x,y)$ to the CM equation $4q_k(x)-t_k(x)^2 = Dy^2$, for some small positive discriminant $D\in\mathbb{Z}$.

Miyaji, Nakabayashi, and Takano \cite{mnt} characterized all families of ordinary prime-order elliptic curves with embedding degrees $k=3,4,6$. For these embedding degrees, the cyclotomic polynomial is quadratic, and the CM equation can be transformed into a generalized Pell equation. These families are parametrized by the polynomials in \cref{table:mnt}. We refer to elliptic curves belonging to the MNT families in \cref{table:mnt} as \emph{MNT curves}. 

\begin{table}[h]
\caption{MNT curves.}
\label{table:mnt}
\begin{center}
\begin{tabular}{|c|c|c|c|}
\hline
$k$ & $q_k(x)$ & $n_k(x)$ & $t_k(x)$ \\ \hline
$3$ & $12x^2 -1$ & $12x^2- 6x +1$ & $6x -1$  \\ \hline
$4$ & $x^2 + x + 1$ & $x^2+2x+2,\,x^2+1$ & $-x,\,x+1$ \\ \hline
$6$ & $4x^2 + 1$ & $4x^2+ 2x +1$ & $-2x +1$  \\\hline
\end{tabular}
\end{center}
\end{table}

For other embedding degrees, there is no analogous characterization of all elliptic curves with a given embedding degree. Moreover, there is currently no method to construct families of prime-order elliptic curves of arbitrary embedding degrees. (If we allow for composite orders, there are algorithms to construct elliptic curves of arbitrary embedding degrees \cite{cp, dem}.) There are two other constructions of prime-order families, stated below.

Freeman \cite{freeman10} has constructed a family of prime-order elliptic curves with $k=10$, which is parametrized by the following polynomials:
\begin{subequations}
\label{eqn:freeman}
\begin{align}
q_{10}(x) &= 25x^4 + 25x^3 + 25x^2 + 10x + 3\,,\\
n_{10}(x) &= 25x^4 + 25x^3 + 15x^2 + 5x + 1\,,\\
t_{10}(x) &= 10x^2 + 5x + 3\,.
\end{align}
\end{subequations}

Barreto and Naehrig \cite{bn} have another construction with $k=12$, parametrized by
\begin{subequations}
\label{eqn:bn}
\begin{align} 
q_{12}(x) &= 36x^4 + 36x^3 + 24x^2 + 6x+1\,,\\
n_{12}(x) &= 36x^4 + 36x^3 + 18x^2 + 6x+1\,,\\
t_{12}(x) &= 6x^2 + 1\,.
\end{align}
\end{subequations}

Other constructions of pairing-friendly elliptic curves have \emph{composite} orders. These include the families of Brezing and Weng \cite{brezingweng} and of Barreto, Lynn, and Scott \cite{bls}.

\section{Cycles of pairing-friendly elliptic curves}

In this paper we study cycles of pairing-friendly elliptic curves. This notion was introduced in \cite{snarks} for applications in cryptography. We re-state \cref{def:cycle} below.

\begin{definition}
\label{def:cycle1}
An \emph{$m$-cycle of elliptic curves} is a list of $m$ distinct elliptic curves $E_1/\mathbb{F}_{q_1}$,$\ldots$, $E_m/\mathbb{F}_{q_m}$, where $q_1,\ldots, q_m$ are prime, such that the numbers of points on these curves satisfy 
\begin{align} 
\#E_1(\mathbb{F}_{q_1}) = q_2\,, \ldots,
\#E_i(\mathbb{F}_{q_i}) = q_{i+1}\,,\ldots,
\#E_m(\mathbb{F}_{q_m}) = q_1\,. \end{align}
\end{definition}

Cryptographic applications require curves in the cycle to have small embedding degree.

\begin{definition}
\label{def:cycle2}
A \emph{$(k_1,\dots,k_m)$-cycle} is an $m$-cycle of distinct ordinary
elliptic curves $E_1/\mathbb{F}_{q_1}$, $\ldots$, $E_m/\mathbb{F}_{q_m}$ such that $E_i/\mathbb{F}_{q_i}$ has embedding degree $k_i$, for each $i=1,\ldots,m$. A $(k_1,\ldots,k_m)$-cycle is \emph{pairing-friendly} if all the $k_i$'s are small (recall \cref{def:cycle-pf}).
\end{definition}

An $m$-cycle is a special case of a $(k_1,\ldots,k_m)$-cycle where the $k_i$'s are arbitrary positive integers (or possibly infinity).

If we require that $q_1,\ldots,q_m$ are \emph{distinct} primes, \cref{def:cycle1} is equivalent to the notion of \emph{aliquot cycles} for elliptic curves $E/\mathbb{Q}$ by Silverman and Stange \cite{ss}. An aliquot $m$-cycle for $E/\mathbb{Q}$ is a sequence of distinct primes $(q_1,\ldots,q_m)$ such that $E$ has good reduction at each prime and, if we denote the reduction of $E$ at $q_i$ by $\tilde{E}_{q_i}$, then
\begin{align}
\#\tilde{E}_{q_1}(\mathbb{F}_{q_1}) = q_2\,, \ldots,
\#\tilde{E}_{q_i}(\mathbb{F}_{q_{i}}) = q_{i+1}\,,\ldots,
\#\tilde{E}_{q_m}(\mathbb{F}_{q_m}) = q_1\,.
\end{align}

Given an aliquot $m$-cycle, we can construct an $m$-cycle of elliptic curves by setting $E_i := \tilde{E}_{q_i}$ for each $i$. Conversely, given an $m$-cycle where $q_1,\ldots,q_m$ are distinct, we can construct a curve $E/\mathbb{Q}$ by computing its coefficients via the Chinese Remainder Theorem in such a way that $E$'s reduction at each $q_i$ is $E_i$.

It is known that cycles of arbitrary lengths exist, based just on the Hasse bound and the fact that every trace in the Hasse bound is realized by an elliptic curve \cite{deuring}.

\begin{proposition}[{\cite[Theorem 5.1]{ss}}]
For every $m \geq 1$ there exists an elliptic curve $E/\mathbb{Q}$ with an aliquot $m$-cycle.
\end{proposition}

However, the foregoing result does not take into account the embedding degrees of the curves. In particular, it is not known if pairing-friendly cycles of arbitrary lengths exist.

The focus of this paper is the study of \emph{pairing-friendly} cycles of elliptic curves. This is a significantly more restrictive notion than the aliquot cycles introduced in \cite{ss}, since a random elliptic curve would not have a small embedding degree. Moreover, there are only few known families of prime-order elliptic curves with small embedding degrees (see \cref{section:families} for a list of all such families). Even without the condition that the curves form a cycle, it is already a difficult problem to construct pairing-friendly elliptic curves of prime order.

We list below a few observations that we will use in this paper. First, the lemma below implies that to construct cycles of elliptic curves for applications (where the size of the finite fields tend to be large), we need only consider \emph{ordinary} elliptic curves.

\begin{lemma}
Let $E_1/\mathbb{F}_{q_1}$,$\ldots$, $E_m/\mathbb{F}_{q_m}$ be an $m$-cycle of elliptic curves, where $q_1,\ldots,q_m \geq 5$ are prime. Then all the curves must be ordinary elliptic curves.
\end{lemma}

\begin{proof}
It is known that for any elliptic curve $E/\mathbb{F}_q$ with $q\geq 5$ prime, $E$ is supersingular if and only if
$\# E(\mathbb{F}_q) = q+1$, see for example \cite[Exercise 5.10]{silverman}. Suppose $E_i/\mathbb{F}_{q_i}$ is supersingular for some $i$, then $\#E(\mathbb{F}_{q_i}) = q_i + 1 = q_{i+1}$. But since $q_i$ is prime, $q_i+1$ is even, hence this cannot hold.
\end{proof}

Next, we present a necessary condition for $m$ elliptic curves to form an $m$-cycle. This condition is \emph{not} sufficient as every trace in the Hasse interval can be realized by an elliptic curve \cite{deuring}, hence this condition is not a strong restriction on the curves in the cycle.

\begin{lemma}
\label{lem:sum-of-traces}
Let $E_1/\mathbb{F}_{q_1}$,$\ldots$, $E_m/\mathbb{F}_{q_m}$ be an $m$-cycle of elliptic curves, with traces $t_1,\ldots,t_m$ respectively. Then the sum of their traces satisfies
\begin{align}
t_1 + \cdots + t_m  = m\,.\label{eq:m-cycle-t}
\end{align}
\end{lemma}

\begin{proof}
Let $n_i = \#E_i(\mathbb{F}_{q_i})$, for each $i=1,\ldots,m$. Since the curves form a cycle, we have the constraints $n_1 = q_2$, $\ldots$, $n_i = q_{i+1}$, $\ldots$, $n_m = q_1$. If we sum up these $m$ equations, we get $n_1 + \cdots + n_m = q_1 + \cdots + q_m$. Using the fact that $n_i = q_i + 1 - t_i$, we get $t_1 + \cdots + t_m = m$.
\end{proof}

\section{MNT cycles}
\label{section:MNT}

We consider pairing-friendly cycles consisting of MNT curves (see \cref{table:mnt}), which are the ordinary prime-order elliptic curves of embedding degrees $3,4,6$. For brevity, we use the term \emph{MNT cycles} for cycles where every curve is an MNT curve.

In \cite{kt, snarks}, MNT curves were used to give the first construction of pairing-friendly 2-cycles. In this section, we construct MNT 4-cycles, and characterize the possible MNT cycles.

\begin{proposition}
\label{prop:mnt-4-6-cycles}
All MNT cycles have lengths 2 or 4, and they are either $(6,4)$-cycles or $(6,4,6,4)$-cycles.
\end{proposition}

The proof of this result proceeds in a few steps. First in \cref{lem:mnt-k-3} we show that no curve in an MNT cycle can have embedding degree 3. Then in \cref{lem:mnt-consecutive-k-4,lem:mnt-consecutive-k-6} we show that no two consecutive curves in an MNT cycle can both have embedding degree 4 or 6. Finally we consider MNT cycles with alternating embedding degrees 4 and 6, and we show that these can only have lengths 2 or 4.

\begin{lemma}
\label{lem:mnt-k-3}
Let $E_1/\mathbb{F}_{q_{k_1}(x_1)},\ldots,E_m/\mathbb{F}_{q_{k_m}(x_m)}$ be an MNT cycle, with $x_1,\ldots, x_m\in\mathbb{Z}$ and embedding degrees $k_1,\ldots,k_m\in\{3,4,6\}$. Then none of the embedding degrees can be 3.
\end{lemma}

To show \cref{lem:mnt-k-3}, we make use of the following result.

\begin{lemma}[{\cite[Proposition 2.10]{washington}}]
\label{lem:q-mod-n}
Let $q$ be a prime such that $q\nmid k$. Then $q$ divides $\Phi_k(a)$ for some $a\in\mathbb{Z}$ if and only if $q \equiv 1 \pmod{k}$.
\end{lemma}

\begin{proof}[Proof of \cref{lem:mnt-k-3}]
By \cref{lem:k-phi}, the condition that $E_i/\mathbb{F}_{q_{k_i}(x_i)}$ has embedding degree $k_i$ implies that $n_{k_i}(x_i) \,|\,\Phi_{k_i}(q_{k_i}(x_i))$. Since $n_{k_i}(x_i) = q_{k_{i+1}}(x_{i+1})$, \cref{lem:q-mod-n} implies
\begin{equation}
q_{k_{i+1}}(x_{i+1}) \equiv 1 \pmod{k_i}\,.
\end{equation} 
Suppose that $k_{j} = 3$ for some $j$. From \cref{table:mnt},
\begin{equation}
q_3(x_j) = 12 x_j^2 - 1 \equiv 1 \pmod{k_{j-1}}\,.
\end{equation}
However, this is not possible since $12x_j^2 - 1 \equiv -1 \pmod{3,4,6}$.
\end{proof}

We show that for any MNT cycle, no two consecutive curves can both have embedding degree $4$ or $6$. 

\begin{lemma}
\label{lem:mnt-consecutive-k-4}
Let $E_1/\mathbb{F}_{q_{k_1}(x_1)},\ldots,E_m/\mathbb{F}_{q_{k_m}(x_m)}$ be an MNT cycle, with $x_1,\ldots,x_m\in\mathbb{Z}$. Then no two consecutive curves can both have embedding degree $4$.
\end{lemma}

\begin{proof}
Suppose to the contrary that $k_i = k_{i+1} = 4$ for some $i$. Then $n_4(x_i) = q_4(x_{i+1})$. From \cref{table:mnt}, $q_4(x_{i+1}) = x_{i+1}^2 + x_{i+1} + 1$, and there are two possibilities for $n_4(x_i)$.

Suppose $n_4(x_i) = x_i^2 + 2x_i + 2$. Then $x_i^2 + 2x_i + 2 = x_{i+1}^2 + x_{i+1} + 1$, which implies
\begin{equation}
(x_i + 1)^2 = x_{i+1}(x_{i+1} + 1)\,.
\end{equation}
This is a contradiction if $x_i \neq -1$, since the product of two consecutive nonzero integers is not a square.\footnote{Suppose that for some nonzero $x,y\in\mathbb{Z}$, $x(x+1)=y^2$. If $x>0$, then
$x^2 <y^2<(x+1)^2$, which has no integer solutions for $x,y$. If $x<0$, then $x^2>y^2>(x+1)^2$, which also has no integer solutions for $x,y$.} But if $x_{i}=-1$, then $n_4(x_i)=1$ would not be prime.

Suppose $n_4(x_i) = x_i^2 + 1$. Then $x_i^2 + 1 = x_{i+1}^2 + x_{i+1} + 1$, which implies
\begin{equation}
x_i^2 = x_{i+1}(x_{i+1} + 1)\,.
\end{equation}
This is a contradiction by the same argument as above.
\end{proof}

\begin{lemma}
\label{lem:mnt-consecutive-k-6}
Let $E_1/\mathbb{F}_{q_{k_1}(x_1)},\ldots,E_m/\mathbb{F}_{q_{k_m}(x_m)}$ be an MNT cycle with $x_1,\ldots,
x_m\in\mathbb{Z}$. Then no two consecutive curves can both have embedding degree $6$.
\end{lemma}

\begin{proof}
Suppose to the contrary that $k_i = k_{i+1} = 6$ for some $i$. Then $n_6(x_i) = q_6(x_{i+1})$. From \cref{table:mnt}, $q_6(x_{i+1}) =4x_{i+1}^2 + 1$, and $n_6(x_i) = 4x_i^2 + 2x_i + 1$. Thus $4x_i^2 + 2x_i + 1 = 4x_{i+1}^2 + 1$, which implies
\begin{equation}
2x_i(2x_i + 1) = (2x_{i+1})^2\,.
\end{equation}
This is a contradiction if $x_{i+1}\neq 0$, since the product of two consecutive nonzero integers is not a square. But if $x_{i+1}=0$, then $q_6(x_{i+1}) = 1$ would not be prime.
\end{proof}

We now consider MNT cycles consisting of elliptic curves with alternating
embedding degrees $4$ and $6$.

\begin{lemma}
\label{lem:mnt-4-6}
Let $E_i/\mathbb{F}_{q_4}(x_i)$, $E_{i+1}/\mathbb{F}_{q_6}(x_{i+1})$
be consecutive curves in an MNT cycle. Then $2|x_{i+1}| = |x_i|$ or $2|x_{i+1}| = |x_i + 1|$.
\end{lemma}

\begin{proof}
We have the condition $n_4(x_i) = q_6(x_{i+1})$. By
\cref{table:mnt}, $q_6(x_{i+1}) = 4x_{i+1}^2 + 1$, and there are
two possibilities for $n_4(x_i)$.

If $n_4(x_i) = x_i^2 + 2x_i + 2$, then $x_i^2 + 2x_i+2 = 4x_{i+1}^2 + 1$, which we simplify to $(x_i+1)^2 = (2x_{i+1})^2$. Thus $2|x_{i+1}| = |x_i + 1|$.

If instead $n_4(x_i) = x_i^2 +1$, then $x_i^2 + 1 = 4x_{i+1}^2 + 1$, which we simplify to $x_i^2 = (2x_{i+1})^2$. Thus $2|x_{i+1}| = |x_i|$.
\end{proof}

\begin{lemma}
\label{lem:mnt-6-4}
Let $E_i/\mathbb{F}_{q_6}(x_i)$, $E_{i+1}/\mathbb{F}_{q_4}(x_{i+1})$ be consecutive curves in an MNT cycle. Then $x_{i+1} = 2x_i$.
\end{lemma}

\begin{proof}
We have the condition $n_6(x_i) = q_4(x_{i+1})$. By \cref{table:mnt}, this gives $4x_i^2 + 2 x_i + 1 = x_{i+1}^2 +
x_{i+1} + 1$, or  $2x_i(2x_i + 1) = x_{i+1}(x_{i+1} + 1)$. This implies $x_{i+1} = 2x_i$.
\end{proof}

We now show \cref{prop:mnt-4-6-cycles} that all MNT cycles are $(6,4)$-cycles or $(6,4,6,4)$-cycles.

\begin{proof}[Proof of \cref{prop:mnt-4-6-cycles}]
By \cref{lem:mnt-k-3}, \cref{lem:mnt-consecutive-k-4} and \cref{lem:mnt-consecutive-k-6}, all MNT cycles consist of curves with embedding degrees alternating between $4$ and $6$, and have even lengths.

Let $E_1/\mathbb{F}_{q_6}(x_1),E_2/\mathbb{F}_{q_4}(x_2), \ldots,E_{2m}/\mathbb{F}_{q_4}(x_{2m})$ be an MNT cycle. We first observe that \cref{lem:mnt-4-6} and \cref{lem:mnt-6-4} imply that $|x_1| = |x_3| = \cdots = |x_{2m-1}|$. Thus $q_6(x_1) = q_6(x_3) = \cdots = q_6(x_{2m-1})$. As there are only two possibilities for $n_6(x_1), n_6(x_3),\ldots, n_6(x_{2m-1})$, for the curves to be distinct we must have $m\leq 4$, and if $m=4$ then we must have $x_3 = -x_1$.

Let $x := x_1$. Then \cref{lem:mnt-6-4} implies $x_2 = 2x$. By \cref{lem:mnt-4-6}, either $x_3 = x$, in which case we have a $(6,4)$-cycle, or $x_3=-x$. For the latter case, \cref{lem:mnt-6-4} implies that $x_4 = -2x$, which gives us a $(6,4,6,4)$-cycle.

By substituting the possible parameter values for $x$ into the polynomials in \cref{table:mnt}, we obtain the parametrizations of the possible families of MNT $(6,4)$-cycles in \cref{table:mnt-2-cycles-polys} and $(6,4,6,4)$-cycles in \cref{table:mnt-4-cycles-polys}. These cycles can be constructed by substituting integer values of $x$ and checking if all the $n(x)$'s and $q(x)$'s are prime.
\end{proof}

The MNT $(6,4,6,4)$-cycles in \cref{table:mnt-4-cycles-polys} are unions of two MNT $(6,4)$-cycles. Indeed, the pairs $(E_1,E_2)$ and $(E_3,E_4)$ each form $(6,4)$-cycles. Furthermore, $E_1,E_3$ are defined over the same finite field. Interestingly, these are the only possible MNT 4-cycles, and no longer cycles consisting of distinct elliptic curves can be obtained by taking unions of MNT 2-cycles. 

\begin{table}
\caption{MNT $(6,4)$-cycles.}
\label{table:mnt-2-cycles-polys}
\begin{center} 
\begin{tabular}{|c|c|c|}
\hline
& $E_1$ & $E_2$  \\ \hline
$k$ & $6$ & $4$  \\ \hline
$q(x)$ & $4x^2 +1$ & $4x^2 + 2x+1$ \\
$n(x)$ &  $4x^2 + 2x+1$ & $4x^2+1$  \\ 
$t(x)$ & $- 2x+1$ & $2x+1$ \\ \hline
\end{tabular}
\end{center}
\end{table}

\begin{table}
\caption{MNT $(6,4,6,4)$-cycles.}
\label{table:mnt-4-cycles-polys}
\begin{center} 
\begin{tabular}{|c|c|c|c|c|}
\hline
& $E_1$ & $E_2$ & $E_3$ & $E_4$ \\ \hline
$k$ & $6$ & $4$ & $6$ & $4$ \\ \hline
$q(x)$ & $4x^2 +1$ & $4x^2 +2x+1$ & $4x^2+1$ & $4x^2-2x+1$ \\
$n(x)$ &  $4x^2+2x+1$ & $4x^2+1$ & $4x^2-2x+1$ & $4x^2+1$ \\
$t(x)$ & $-2x+1$ & $2x+1$ & $2x+1$ & $-2x+1$ \\
\hline
\end{tabular}
\end{center}
\end{table}

\begin{example}
\label{ex:mnt-2-cycle}
We give an example of an MNT $(6,4)$-cycle, using the parametrization in \cref{table:mnt-2-cycles-polys}. If $x=1$, we check that $4x^2+1 = 5$ and $4x^2 -2x+1 = 3$ are prime. We compute each of the two curves in the cycle using the CM method and Sage \cite{sagemath}.
\begin{subequations}
\begin{align}
E_1/\mathbb{F}_{5}\,&:\,y^2 = x^3 + 4x + 2\,,\\
E_2/\mathbb{F}_{3}\,&:\,y^2 = x^3 + 2x^2 + 1\,.
\end{align}
\end{subequations}
We list all the points of these curves in \cref{table:mnt-2-cycle-points}.
\end{example}

\begin{example}
We give an example of an MNT $(6,4,6,4)$-cycle, using the parametrization in \cref{table:mnt-4-cycles-polys}. If $x=3$, we check that $4x^2 + 1=37$, $4x^2 +2x+1=43$ and $4x^2-2x+1 = 31$ are all prime. We compute the curves using Sage \cite{sagemath}.
\begin{subequations}
\begin{align}
E_1/\mathbb{F}_{37} \,&:\, y^2 = x^3 + 24x + 16\,,\\
E_2/\mathbb{F}_{43}\,&:\,y^2 = x^3 + 36x + 5\,, \\
E_3/\mathbb{F}_{37}\,&:\, y^2 = x^3 + 22x + 27\,,\\
E_4/\mathbb{F}_{31}\,&:\, y^2 = x^3 + 26x + 21\,.
\end{align}
\end{subequations}
We list all the points of these curves in \cref{table:mnt-4-cycle-points}.
\end{example}

\section{Two-cycles of specific embedding degrees}
\label{section:degree4}

In this section we prove the following result.

\begin{proposition}
\label{prop:degree4-main}
There are no $(5, 10)$-, $(8, 8)$-, or $(12, 12)$-cycles.
\end{proposition}

The pairs $(5, 10),(8, 8),(12, 12)$ are precisely the pairs $(k_1, k_2)$ whose cyclotomic polynomials satisfy $\Phi_{k_1}(x) = \Phi_{k_2}(-x)$ and $\deg \Phi_{k_1}(x) = 4$.  To prove \cref{prop:degree4-main}, we first use these conditions to reduce from the problem of classifying $(k_1, k_2)$-cycles to that of finding integral points on a few quartic curves, with finitely many exceptions, in \cref{lem:degree4-quartics}. We then classify all integral points on these quartic curves and the finitely many exceptions using computational tools, yielding no actual $(k_1, k_2)$-cycles.

Note that in the case of 2-cycles, when we require nontrivial embedding degrees, the two curves cannot have equal field sizes.\footnote{Even when allowed, curves $E/\mathbb{F}_q$ with $q = \sharp E(\mathbb{F}_q)$, known as \textit{anomalous}, are undesirable because discrete logarithms can be computed in polynomial time via the SSSA attack \cite{semaev, smart, satoh_araki}.}

We first prove the following more general result, which we hope will also have applications to other kinds of 2-cycles.

\begin{lemma}
\label{lem:degree4-eq}
Let $(k_1, k_2)$ satisfy $\Phi_{k_1}(x) = \Phi_{k_2}(-x)$.  Let $E_1/\mathbb{F}_{q_1}, E_2/\mathbb{F}_{q_2}$ be a $(k_1, k_2)$-cycle with $q_1 > q_2$, and let $c = q_1 - q_2$.  Then $q_1q_2 \mid \Phi_{k_1}(c)$.  Additionally, for some integer $d$ whose prime divisors are all congruent to $1 \pmod{k_1}$, there is an integer $y$ such that
\begin{equation}
\label{eq:degree4-plane-curve}
y^2 = c^2d^2 + 4d\Phi_{k_1}(c)\,.
\end{equation} 
\end{lemma}

\begin{proof}
By \cref{lem:k-phi}, the condition that $E_1/\mathbb{F}_{q_1}$ has embedding degree $k_1$ implies that $q_2 \mid \Phi_{k_1}(q_1)$.  Then $q_2 \mid \Phi_{k_1}(q_1 - q_2)$ as well.  Similarly, $q_1 \mid \Phi_{k_2}(q_2 - q_1) = \Phi_{k_1}(q_1 - q_2)$.  It follows that $q_1q_2 \mid \Phi_{k_1}(q_1 - q_2) = \Phi_{k_1}(c)$ as $q_1$ and $q_2$ are distinct primes.

Then $dq_1q_2 = \Phi_{k_1}(c)$ for some integer $d$. Using $q_1 = q_2 + c$, we can rewrite this as
\begin{equation}
dq_2^2 + cdq_2 - \Phi_{k_1}(c) = 0\,.
\end{equation}
For this quadratic equation in $q_2$ to have an integral solution, the discriminant
\begin{equation}
c^2d^2 + 4d\Phi_{k_1}(c)
\end{equation}
must be a perfect square, so that there is a $y$ satisfying \cref{eq:degree4-plane-curve}.

Also, for any prime $p \mid d$, the above relation $dq_1q_2 = \Phi_{k_1}(c)$ implies that $p \mid \Phi_{k_1}(c)$.  Hence $p \equiv 1 \pmod{k_1}$ by \cref{lem:q-mod-n}.
\end{proof}

\begin{lemma}
\label{lem:degree4-quartics}
In the situation of \cref{lem:degree4-eq}, additionally let $\deg \Phi_{k_1}(x) = 4$.  Equivalently, let $(k_1, k_2) \in \{(5, 10), (8, 8), (10, 5), (12, 12)\}$.  Then $c \le 82$ or $1 \le d \le 16$.
\end{lemma}

\begin{proof}
Let $c \ge 83$.  Then $\Phi_{k_1}(c)> 0$, so the relation $dq_1q_2 = \Phi_{k_1}(c)$ implies $d \ge 1$.  Next, because $E_2/\mathbb{F}_{q_2}$ has $q_1$ points, the Hasse bound implies $|q_1 - (q_2 + 1)| \le 2\sqrt{q_2}$.  Substituting $c = q_1 - q_2$ and rearranging shows $q_2 \ge (c-1)^2/4$.  The same holds for $q_1$ since $q_1 > q_2$.  Then $dq_1q_2 = \Phi_{k_1}(c)$ implies
\begin{equation}
\label{eqn:bound-d}
d < 16\frac{\Phi_{k_1}(c)}{(c-1)^4} \,.
\end{equation}
For each $k_1 \in \{5, 8, 10, 12\}$, we find that for $c \ge 83$, the right-hand side is at most 17.  Thus either $c \le 82$ or $1 \le d \le 16$.
\end{proof}

For each $(k_1, k_2)$ listed in \cref{lem:degree4-quartics}, using the fact $q_1, q_2 \mid \Phi_{k_1}(c)$ from \cref{lem:degree4-eq}, one can see that the case $c \le 82$ yields only finitely many $(k_1, k_2)$-cycles.  Also, for each $1 \le d \le 16$ whose prime divisors are congruent to $1 \pmod{k_1}$, one can show that \cref{eq:degree4-plane-curve} defines a plane curve of genus 1 in the coordinates $(c, y)$.  Siegel's Theorem \cite[Theorem 8.2.4]{lang_diophantine_geometry} implies that such a curve has only finitely many integral points, hence there are only finitely many $(k_1, k_2)$-cycles.

We now use computational tools to show that there are in fact no $(k_1, k_2)$-cycles.

\begin{proof}[Proof of \cref{prop:degree4-main}]
Using the fact $q_1, q_2 \mid \Phi_{k_1}(c)$ from \cref{lem:degree4-eq}, it is easy to enumerate all $(k_1, k_2)$-cycles which have $c \le 82$, for $(k_1, k_2) \in \{(5, 10), (8, 8), (10, 5), (12, 12)\}$.  Doing so using Sage \cite{sagemath} reveals no such examples.

We now consider the case $d \le 16$.  Restricting to values of $d$ whose prime factors are all congruent to $1 \pmod{k_1}$, we are left with the cases shown in \cref{table:degree4-main-cases}.

\begin{table}[h]
\caption{Cases $((k_1, k_2), d)$ satisfying \cref{lem:degree4-quartics} when $c \ge 83$.}
\label{table:degree4-main-cases}
\begin{center}
\begin{tabular}{|c|c|}
\hline
$(k_1, k_2)$ & $d$ \\ \hline
$(5, 10)$ & 11 \\ \hline
$(10, 5)$ & 13\\ \hline
$(12, 12)$ & 13\\ \hline
\end{tabular}
\end{center}
\end{table}

In the case $(k_1, k_2) = (12, 12)$, $d = 13$, we can enumerate the integral points of \cref{eq:degree4-plane-curve} using Magma's \texttt{IntegralQuarticPoints} function \cite{magma}.  Doing so gives no examples with $c \ge 83$.
  
When $(k_1, k_2) = (5, 10)$ or $(10, 5)$ and $d = 11$, Sage \cite{sagemath} finds that \cref{eq:degree4-plane-curve} has no solutions over the ring of integers modulo 16, hence it has no integral solutions.  Thus these cases also give no examples.
\end{proof}

When $\deg \Phi_{k_1}(x) > 4$, the bound on $d$ in \cref{eqn:bound-d} no longer converges to a finite value as $c \rightarrow \infty$, so we cannot reduce to finding integral points on a finite number of curves as above.  It would be interesting to find more general arguments which work for higher-degree cyclotomic polynomials, such as the case of $(16, 16)$-cycles, where $\Phi_{k_1}(x) = \Phi_{k_2}(x) = x^8 + 1$.

\section{Cycles with the same discriminant}
\label{section:discriminant}

In this section we show that if we construct cycles from elliptic curves of the same discriminant $D$, then the length of the cycle must be small. This implies that to construct elliptic curves from polynomial families, we cannot use families with a fixed discriminant. The results in this section are \emph{independent} of the embedding degrees of the elliptic curves.

We first show that any 2-cycle of ordinary elliptic curves consists of curves with the same
discriminant.

\begin{proposition}
\label{prop:2cyclecm}
Let $E_1/\mathbb{F}_{q_1}, E_2/\mathbb{F}_{q_2}$ be a 2-cycle of ordinary elliptic curves. Then they both have the same discriminant for complex multiplication.
\end{proposition}

\begin{proof}
Let $t_i$ be the trace of $E_i$ for each $i$. Then $q_2 = q_1 + 1 -t_1$ and $q_1 = q_2 + 1 -t_2$. This implies $t_1+t_2=2$, and
\begin{align*}
4q_2 - t_2^2 = 4(q_1+1-t_1) - (2-t_1)^2 = 4q_1-t_1^2\,.
\end{align*}
The discriminant of $E_i$ is the squarefree part of $4q_i-t_i^2$, so the two curves have the same discriminant.
\end{proof}

The converse is also true if $D>3$, as shown in \cite[Corollary 6.2]{ss} and \cite[Theorem 3.4]{reciprocity}. We present an adapted version of the proof below.

\begin{proposition}
\label{prop:cyclecm}
Let $D>3$ be a squarefree integer such that $-D\equiv 0,1\pmod{4}$. Suppose that we have an $m$-cycle of ordinary elliptic curves $E_1/\mathbb{F}_{q_1},\ldots, E_m/\mathbb{F}_{q_m}$ such that each elliptic curve has discriminant $D$ and $q_1,\ldots,q_m$ are distinct primes. Then $m\leq 2$.
\end{proposition}

\begin{proof}
For each $i=1,\ldots,m$, let $y_i\in\mathbb{Z}$ be such that the CM equation $4q_i - t_i^2 = Dy_i^2$ is satisfied. Firstly, we note that if we fix $q_i$ and $D$, the solution $(t_i,y_i)$ to the CM equation is unique up to sign. This follows from the fact that, under our assumptions on $D$, the units in the ring of integers of $\mathbb{Q}(\sqrt{-D})$ are $\pm 1$, hence if two elements have the same norm, then they differ by a multiple of $\pm 1$.

Now let $E_i/\mathbb{F}_{q_i}$, $E_{i+1}/\mathbb{F}_{q_{i+1}}$ be two consecutive curves in the cycle. Since 
\begin{align}
4q_{i+1} - (t_i-2)^2 = 4(q_{i+1}-1+t_i) - t_i^2 = 4q_i - t_i^2 = Dy_i^2\,,
\end{align}
thus $t_i-2 = \pm t_{i+1}$, and $y_i = \pm y_{i+1}$, by the uniqueness of the solution to the CM equation. 

Suppose that $m\geq 3$. Without loss of generality, assume that $q_2$ is the smallest prime in the cycle. Then $q_{2} < q_1,q_3$. From the previous paragraph we also have $t_1-2 = \pm t_{2}$. We consider the two cases separately.

If $t_1-2 = t_{2}$, then $q_{2} = q_1 - 1-t_{2}$. So we have the inequalities $q_1 = q_{2} + 1 + t_{2} >q_{2}$, and $q_{3} = q_{2} + 1 -t_{2} > q_{2}$. Hence $1>t_{2} > -1$ so $t_{2}=0$. But this implies that $q_1=q_{3}$, which contradicts the assumption that the $q_i$'s are distinct.

If $t_1-2 = -t_{2}$, then $q_{2} = q_1 -1 + t_{2}$, so $q_1 = q_{2} + 1 -t_{2} = q_{3}$. This again contradicts the assumption that the $q_i$'s are distinct.
\end{proof}

For the case where $D=3$, we cite the following result from \cite{reciprocity}.

\begin{proposition}[{\cite[Theorem 3.4]{reciprocity}}]
Suppose that we have an $m$-cycle of ordinary elliptic curves $E_1/\mathbb{F}_{q_1},\ldots, E_m/\mathbb{F}_{q_m}$ such that each elliptic curve has discriminant $D$ and $q_1,\ldots,q_m$ are distinct primes. If $m\geq 3$, then $m=6$ and $D=3$.
\end{proposition}

The results in this section show that to construct $m$-cycles of elliptic curves with a fixed discriminant $D$, either $m\leq 2$ or $m=6$ and $D=3$. This places a strong restriction on possible cycles, and implies that we cannot construct long cycles from a single family of elliptic curves with a fixed discriminant. For example, the Barreto--Naehrig curves \cite{bn} all have discriminant $D=3$.

We also note that the results in this section do not depend on the embedding degrees of the elliptic curves. It remains an open question to understand how restricting the embedding degrees places further restrictions on the possible cycles.

\section{Cycles with cofactors}
\label{section:cofactor}

Allowing for non-prime orders gives greater flexibility in constructing elliptic curves, while still having relevance to cryptographic applications. While there are few embedding degrees that can be achieved by current constructions of prime-order curves, there are methods that achieve \emph{arbitrary} embedding degrees for composite-order curves \cite{cp, dem}. While composite-order curves tend to be less preferable than prime-order curves in applications, they can still be practical and sometimes even preferable.%
\footnote{
For example, Barreto--Lynn-Scott curves \cite{bls} are composite-order curves that, thanks to their high embedding degrees, enable efficient implementations at high-security levels. As another example, Edwards curves \cite{Edwards07,BernsteinL07,BernsteinBJLP08} are composite-order curves that, thanks to their complete formulas for addition, enable efficient implementations that resist various side channels (e.g., \cite{BernsteinDLSY11}).
}

Nevertheless, we show in this section that allowing for non-prime orders does not give us greater flexibility in constructing cycles. Our arguments in this section rely only on the Hasse bound and the constraints on the orders of the elliptic curves posed by the cycle condition. 

\begin{definition}
\label{def:generalized-cycle}
An \emph{$m$-cycle of elliptic curves with cofactors} consists of $m$ distinct elliptic curves $E_1/\mathbb{F}_{q_1}$,$\ldots$, $E_m/\mathbb{F}_{q_m}$ such that for positive integer cofactors $h_1,\ldots,h_m$,
\begin{align}
\#E_1(\mathbb{F}_{q_1}) = h_1 q_2\,,\ldots,
\#E_i(\mathbb{F}_{q_i}) = h_i q_{i+1}\,,\ldots,
\#E_m(\mathbb{F}_{q_m}) = h_m q_1\,. \end{align}
\end{definition} 

If all the cofactors are 1, then \cref{def:generalized-cycle} reduces to \cref{def:cycle}. We show that, for any $m>1$, we cannot have $m$-cycles of elliptic curves with any nontrivial cofactor (and large orders). We deduce this by considering only the Hasse bound on the orders of the curves.

\begin{proposition}
For all $m>1$, there exists no $m$-cycle of elliptic curves having at least one nontrivial cofactor (greater than $1$), if $q_1,\ldots,q_m>12m^2$.
\end{proposition}

\begin{proof}
We first prove this for the simpler case where $m=2$. Suppose that we
have a 2-cycle of elliptic curves $E_1/\mathbb{F}_{q_1}, E_2/\mathbb{F}_{q_2}$ with cofactors such that $\#E_1(\mathbb{F}_{q_1})= h_1q_2$, $\#E_2(\mathbb{F}_{q_2}) = h_2q_1$. The Hasse bound for $E_1$ implies
\begin{align}
q_1 + 1 - 2\sqrt{q_1} \leq h_1 q_2 \leq q_1 + 1 + 2\sqrt{q_1}\,.
\end{align}
We can express this as $(\sqrt{q_1} -1)^2 \leq h_1 q_2 \leq (\sqrt{q_1}+1)^2$. Applying the same argument to $E_2$, we get the following two inequalities
\begin{align}
\sqrt{q_1} - 1 &\leq \sqrt{h_1 q_2} \leq \sqrt{q_1} + 1\,,\\
\sqrt{q_2} - 1 &\leq \sqrt{h_2 q_1} \leq \sqrt{q_2} + 1\,.
\end{align}
We can then bound $q_2$ as follows
\begin{align}
\sqrt{h_1 q_2} \leq \frac{1}{\sqrt{h_2}} (\sqrt{q_2} + 1) + 1\,.
\end{align}
If $h_1>1$ or $h_2>1$, this implies that
\begin{align}
     \sqrt{q_2} \leq \frac{\sqrt{h_2} + 1}{\sqrt{h_1h_2} - 1}
\leq 1 + \frac{2}{\sqrt{h_1h_2} -1} < 3\,.
\end{align}
The same argument applies for bounding $q_1$. Hence for any 2-cycle with nontrivial cofactors, the elliptic curves must have small orders. 

We now extend the argument above to $m$-cycles with cofactors, for all $m>2$. Suppose we have an $m$-cycle with cofactors $\#E_1(\mathbb{F}_{q_1}) = h_1 q_2\,,\#E_2(\mathbb{F}_{q_2}) = h_2q_3\,,\ldots,\#E_m(\mathbb{F}_{q_m}) = h_m q_1$. Applying the same argument as before, we have the inequalities
\begin{subequations}
\begin{align}
\sqrt{h_m q_1} &\leq \sqrt{q_m} + 1 \\
& \leq \frac{1}{\sqrt{h_{m-1}}} (\sqrt{q_{m-1}} + 1) + 1 \nonumber \\
& \leq \frac{1}{\sqrt{h_{m-1} h_{m-2}}} \left(\sqrt{q_{m-2}} + 1 \right) + \left( 1 + \frac{1}{\sqrt{h_{m-1}}}
  \right) \nonumber \\
& \vdots \nonumber\\
& \leq \frac{1}{\sqrt{h_{m-1} \cdots h_1}} \left(\sqrt{q_1} + 1\right) +
    \left( 1 + \frac{1}{\sqrt{h_{m-1}}} + \cdots +
    \frac{1}{\sqrt{h_{m-1}\cdots h_2}}\right)\,.
\end{align}
\end{subequations}
We simplify this to
\begin{align}
\sqrt{q_1} \left(1-\frac{1}{\sqrt{h_m\cdots h_1}}\right) &\leq
\frac{1}{\sqrt{h_m}} + \frac{1}{\sqrt{h_m h_{m-1}}} + \cdots +
\frac{1}{\sqrt{h_m\cdots h_1}} \,.
\end{align}
If at least one of $h_1,\ldots, h_m$ is greater than 1, then we can bound $q_1$ as follows.
\begin{align}
     \sqrt{q_1} \leq \frac{m}{1-\frac{1}{\sqrt{h_m\cdots h_1}}}
\leq \frac{m}{1-\frac{1}{\sqrt{2}}} = (2+\sqrt{2})m \,.
\end{align}
The above argument applies for $q_2,\ldots,q_m$, hence $q_i \leq (2+\sqrt{2})^2 m^2 < 12m^2$ for each $i$. For cryptographic applications, we would require the elliptic curves to be defined over much larger fields than the size of the cycle, contrary to this bound. 
\end{proof}

\section{Other cycles on parametrized families}
\label{section:other}

We have shown in \cref{section:cofactor} that it is not possible to construct cycles of elliptic curves with nontrivial cofactors (and large orders relative to cycle length). Hence cycles of elliptic curves must be assembled from \emph{prime}-order elliptic curves. At present the only known families of \emph{pairing-friendly} prime-order elliptic curves are the MNT curves for $k=3,4,6$, Freeman curves for $k=10$ \cite{freeman10}, and Barreto--Naehrig curves for $k=12$ \cite{bn}. Now we prove that we cannot construct cycles from just Freeman curves or from just Barreto--Naehrig curves.

\begin{proposition}
There do not exist cycles consisting only of Freeman curves.
\end{proposition}

\begin{proof}
\cref{lem:sum-of-traces} poses a restriction on the sum of the traces in a cycle. The trace of the Freeman curves is parametrized by $t(x) = 10x^2 + 5x + 3$ (see \cref{eqn:freeman}). We note that $t(x)>1$ for all $x\in\mathbb{R}$, since the discriminant of $t(x)-1$ is $-55$. Hence the condition in \cref{lem:sum-of-traces} cannot be satisfied for cycles consisting only of Freeman curves. 
\end{proof}

\begin{proposition}
There do not exist cycles consisting only of Barreto--Naehrig curves.
\end{proposition}

\begin{proof}
We again use \cref{lem:sum-of-traces}. The trace of the Barreto--Naehrig curves is parametrized by $t(x) = 6x^2+1$ (see \cref{eqn:bn}), hence if we have a family of elliptic curves consisting only of Barreto--Naehrig curves, then each trace has to be $1$. So $x=0$ and $q(x)=n(x)=1$ for every curve in the cycle, which is impossible since $q(x)$ and $n(x)$ have to be prime.
\end{proof}

We remark that the proof of \cref{lem:mnt-k-3} also shows that there do not exist cycles consisting of just Barreto--Naehrig curves and MNT curves of embedding degree 3.

For combinations of MNT, Freeman, and Barreto--Naehrig curves, we did a preliminary investigation using Gr\"obner bases to find solutions to the following system of polynomial equations in $m$ variables $x_1,\ldots,x_m$, where $k_1,\ldots,k_m\in\{3,4,6,10,12\}$.
\begin{align}
n_{k_1}(x_1) = q_{k_2}(x_2)\,,
n_{k_2}(x_2) = q_{k_3}(x_3),
\ldots,
n_{k_m}(x_m) = q_{k_1}(x_1)\,.
\end{align}
For $m\leq 4$ we found that the ideals generated by these polynomials have dimension $0$ apart from the MNT cycles in \cref{prop:mnt-4-6-cycles}, implying that we cannot construct other families of cycles of length up to $4$. We leave it as an open problem to construct cycles from combinations of these families, or to show that they do not exist.

\appendix
\section{Dual Elliptic Primes}
\label{appendix:dualelliptic}

In \cite{preda, dualelliptic}, dual elliptic primes were introduced for applications in primality proving. 

\begin{definition}[{\cite[Definition 10]{dualelliptic}}]
Two primes $p,q$ are \emph{dual elliptic primes associated to an order $\mathcal{O}\subseteq \mathbb{Q}(\sqrt{-D})$} if there is a prime $\pi\in\mathcal{O}$ such that $p=\pi\overline{\pi}$ and $q=(\pi+\varepsilon)\overline{(\pi+\varepsilon)}$ with $\varepsilon = \pm 1$. 
\end{definition}

Dual elliptic primes are equivalent to 2-cycles of ordinary elliptic curves. 

\begin{proposition}
Let $p,q$ be dual elliptic primes associated to an order
$\mathcal{O}\subseteq\mathbb{Q}(\sqrt{-D})$. Then $p,q$ correspond
bijectively to a 2-cycle of ordinary elliptic curves $E_1/\mathbb{F}_p, E_2/\mathbb{F}_q$ with complex multiplication by $\mathcal{O}$.
\end{proposition}

\begin{proof}
Let $p,q$ be dual elliptic primes. Then
\begin{align}
q = (\pi+\varepsilon)\overline{(\pi+\varepsilon)} =
p + \varepsilon(\pi+\overline{\pi}) + 1\,.
\end{align}
Let $t_1 = -\varepsilon(\pi+\overline{\pi})$ and $t_2
= \varepsilon((\pi+\varepsilon)+\overline{(\pi+\varepsilon)})$. Then
$q = p-t_1 +1$ and $t_2 = -t_1 +2$, so $p = q+1-t_2$. Thus the
elliptic curves $E_1/\mathbb{F}_p, E_2/\mathbb{F}_q$ with traces
$t_1,t_2$ respectively form a 2-cycle. Moreover, since $p,q$ are the norms of principal $\mathcal{O}$-ideals, these elliptic curves have complex multiplication by $\mathcal{O}$.

Conversely, let $E_1/\mathbb{F}_p, E_2/\mathbb{F}_q$ be a 2-cycle of ordinary elliptic curves with complex multiplication by $\mathcal{O}$. We can write the CM equations as $4p - t_1^2 = y^2D$ and $4q - t_2^2 = y^2D$. Let $\lambda_1$ be a root of $x^2-t_1x+p$, and let $\lambda_2$ be a root of $x^2 - t_2x+q$, chosen such that
\begin{align}
\lambda_1 &= \frac{t_1+
  y\sqrt{-D}}{2}\,,\\
\lambda_2 &= \frac{t_2- y\sqrt{-D}}{2}\,.
\end{align}

Then $p,q$ are the norms of the principal $\mathcal{O}$-ideals $(\lambda_1), (\lambda_2)$ respectively. Since $t_2 = 2-t_1$, we have $\lambda_2 = 1-\lambda_1$. Let $\varepsilon = \pm 1$ and $\pi = -\varepsilon\lambda_1$. Then $p = \lambda_1 \overline{\lambda_1} = \pi \overline{\pi}$, and $q=\lambda_2 \overline{\lambda_2} = (\pi+\varepsilon)\overline{(\pi+\varepsilon)}$.
\end{proof}

\newpage

\begin{table}
\caption{Example of an MNT $(6,4)$-cycle.}
\label{table:mnt-2-cycle-points}
\footnotesize
\begin{center} 
\begin{tabular}{|c|c|c|c|}
\hline
& $E_1$ & $E_2$ \\ \hline
& $y^2 = x^3 + 4x + 2$ & $y^2 = x^3 + 2x^2 + 1$   \\ \hline
$(q,n,t,k,D)$ & $(5,3,3,6,11)$ & $(3,5,-1,4,11)$   \\ \hline
\begin{minipage}{2.5cm} \centering
points (excluding point at infinity)
\end{minipage} & 
\begin{minipage}{2cm} \centering
(3,1)\\ (3,4)
\end{minipage} &
\begin{minipage}{2cm} \centering
(0,1)\\ (0,2)\\ (1,1)\\ (1,2)
\end{minipage} \\ \hline
\end{tabular}
\end{center}
\end{table}

\begin{table}
\caption{Example of an MNT $(6,4,6,4)$-cycle.}
\label{table:mnt-4-cycle-points}
\footnotesize
\begin{center} 
\begin{tabular}{|c|c|c|c|c|c|}
\hline
 & $E_1$ & $E_2$ & $E_3$ & $E_4$ \\ \hline
&  $y^2 = x^3 + 24x + 16$ & $y^2 = x^3 + 36x + 5$ & $y^2 = x^3 + 22x + 27$ & $y^2 = x^3 + 26x +
21$  \\ \hline
$(q,n,t,k,D)$ &  $(37,43,-5,6,123)$ & $(43,37,7,4,123)$ & $(37,31,7,6,11)$ & $(31,37,-5,4,11)$ \\ \hline
\begin{minipage}{2.1cm} \centering
points (excluding point at infinity)
\end{minipage}
& 
 \begin{minipage}{1cm} \centering
(0,4)\\ (0,33)\\ (1,2)\\ (1,35)\\ (3,2)\\ (3,35)\\
(4,18)\\ (4,19)\\ (7,3)\\ (7,34)\\ (9,6)\\ (9,31)\\ (12,16)\\
(12,21)\\ (13,3)\\ (13,34)\\ (14,5)\\ (14,32)\\ (17,3)\\ (17,34)\\(18,8)\end{minipage}
\begin{minipage}{1cm} \centering
(18,29)\\ (23,9)\\ (23,28)\\ (26,7)\\ (26,30)\\ (27,16)\\ (27,21)\\ (28,12)\\ (28,25)\\ (31,10)\\ (31,27)\\ (32,17)\\ (32,20)\\ (33,2)\\ (33,35)\\ (34,18)\\ (34,19)\\ (35,16)\\ (35,21)\\ (36,18)\\ (36,19)
\end{minipage} &
\begin{minipage}{1cm} \centering
(3,21)\\ (3,22)\\ (4,16)\\ (4,27)\\ (5,3)\\ (5,40)\\
(7,16)\\ (7,27)\\ (8,17)\\ (8,26)\\ (12,12)\\ (12,31)\\ (13,2)\\
(13,41)\\ (18,11)\\ (18,32)\\ (19,18)\\ (19,25)\end{minipage}
\begin{minipage}{1cm} \centering
(23,10)\\ (23,33)\\(29,5)\\ (29,38)\\ (30,7)\\ (30,36)\\ (31,9)\\ (31,34)\\ (32,16)\\ (32,27)\\ (33,8)\\ (33,35)\\ (38,1)\\ (38,42)\\ (41,21)\\ (41,22)\\ (42,21)\\ (42,22)\end{minipage}
& \begin{minipage}{1cm} \centering
(0,8)\\ (0,29)\\ (3,3)\\ (3,34)\\ (5,15)\\ (5,22)\\
(8,7)\\ (8,30)\\ (10,10)\\ (10,27)\\ (11,3)\\ (11,34)\\ (12,13)\\
(12,24)\\ (23,3)\end{minipage}
\begin{minipage}{1cm} \centering
(23,34)\\ (25,12)\\ (25,25)\\ (27,18)\\ (27,19)\\ (28,5)\\ (28,32)\\ (30,14)\\ (30,23)\\ (31,7)\\ (31,30)\\ (35,7)\\ (35,30)\\ (36,2)\\ (36,35)
\end{minipage}
& \begin{minipage}{1cm} \centering
(2,9)\\ (2,22)\\ (3,8)\\ (3,23)\\ (5,11)\\ (5,20)\\
(7,9)\\ (7,22)\\ (8,11)\\ (8,20)\\ (10,14)\\ (10,17)\\ (13,13)\\
(13,18)\\ (15,2)\\ (15,29)\\ (16,10)\\ (16,21)\end{minipage}
\begin{minipage}{1cm} \centering
(18,11)\\(18,20)\\ (20,4)\\ (20,27)\\ (21,1)\\ (21,30)\\ (22,9)\\ (22,22)\\ (23,13)\\ (23,18)\\ (26,13)\\ (26,18)\\ (27,15)\\ (27,16)\\ (28,3)\\ (28,28)\\ (30,5)\\ (30,26)
\end{minipage} \\ \hline
\end{tabular}
\end{center}
\end{table}

\ifsiaga\else
\clearpage
\fi
\section*{Acknowledgements}

We thank Bernd Sturmfels for helpful comments. We also thank
Fran\c{c}ois Morain for informing us about the application of cycles
of elliptic curves to primality proving. We thank Pierre Houedry for
pointing out an error in a previous version of our paper. This work was supported by the UC Berkeley Center for Long-Term Cybersecurity and a UC Berkeley University Fellowship.

\ifsiaga
\bibliographystyle{siamplain}
\else
\bibliographystyle{alpha}
\fi
\bibliography{ref}

\newcommand{\etalchar}[1]{$^{#1}$}
\begin{thebibliography}{BBC{\etalchar{+}}12}

\bibitem[BBC{\etalchar{+}}12]{reciprocity}
L.~{Babinkostova}, K.~M. {Bombardier}, M.~M. {Cole}, T.~A. {Morrell}, and C.~B.
  {Scott}.
\newblock {Elliptic Reciprocity}.
\newblock {\em ArXiv e-prints}, December 2012.

\bibitem[BBJ{\etalchar{+}}08]{BernsteinBJLP08}
Daniel~J. Bernstein, Peter Birkner, Marc Joye, Tanja Lange, and Christiane
  Peters.
\newblock Twisted {E}dwards curves.
\newblock In {\em Proceedings of the 1st International Conference on Cryptology
  in Africa}, AFRICACRYPT' 08, pages 389--405, 2008.

\bibitem[BCP97]{magma}
Wieb Bosma, John Cannon, and Catherine Playoust.
\newblock The {M}agma algebra system. {I}. {T}he user language.
\newblock {\em J. Symbolic Comput.}, 24(3-4):235--265, 1997.
\newblock Computational algebra and number theory (London, 1993).

\bibitem[BCTV14]{snarks}
Eli {Ben-Sasson}, Alessandro Chiesa, Eran Tromer, and Madars Virza.
\newblock Scalable zero knowledge via cycles of elliptic curves.
\newblock In {\em Proceedings of the 34th Annual International Cryptology
  Conference}, CRYPTO~'14, pages 276--294, 2014.
\newblock Extended version at \url{http://eprint.iacr.org/2014/595}.

\bibitem[BDL{\etalchar{+}}11]{BernsteinDLSY11}
Daniel~J. Bernstein, Niels Duif, Tanja Lange, Peter Schwabe, and Bo-Yin Yang.
\newblock High-speed high-security signatures.
\newblock In {\em Proceedings of the 13th International Conference on
  Cryptographic Hardware and Embedded Systems}, CHES '11, pages 124--142, 2011.

\bibitem[BF03]{BonehF03}
Dan Boneh and Matthew~K. Franklin.
\newblock Identity-based encryption from the {W}eil pairing.
\newblock {\em {SIAM} Journal on Computing}, 32(3):586--615, 2003.

\bibitem[BL07]{BernsteinL07}
Daniel~J. Bernstein and Tanja Lange.
\newblock Faster addition and doubling on elliptic curves.
\newblock In {\em Proceedings of the 13th International Conference on the
  Theory and Application of Cryptology and Information Security}, ASIACRYPT
  '07, pages 29--50, 2007.

\bibitem[BLS02]{bls}
Paulo~SLM Barreto, Ben Lynn, and Michael Scott.
\newblock Constructing elliptic curves with prescribed embedding degrees.
\newblock In {\em International Conference on Security in Communication
  Networks}, pages 257--267. Springer, 2002.

\bibitem[BN05]{bn}
Paulo~SLM Barreto and Michael Naehrig.
\newblock Pairing-friendly elliptic curves of prime order.
\newblock In {\em International Workshop on Selected Areas in Cryptography},
  pages 319--331. Springer, 2005.

\bibitem[Bre17]{Breitman17}
Arthur Breitman.
\newblock Scaling {T}ezos.
\newblock \url{https://hackernoon.com/scaling-tezo-8de241dd91bd}, 2017.

\bibitem[BW05]{brezingweng}
Friederike Brezing and Annegret Weng.
\newblock Elliptic curves suitable for pairing based cryptography.
\newblock {\em Designs, Codes and Cryptography}, 37(1):133--141, 2005.

\bibitem[Cod18]{coda}
Coda.
\newblock Coda cryptocurrency protocol.
\newblock \url{https://codaprotocol.com/}, 2018.

\bibitem[CP01]{cp}
C.~Cocks and R.G.E. Pinch.
\newblock Identity-based cryptosystems based on the {W}eil pairing.
\newblock Unpublished manuscript, 2001.

\bibitem[DEM05]{dem}
R{\'e}gis Dupont, Andreas Enge, and Fran{\c{c}}ois Morain.
\newblock Building curves with arbitrary small {MOV} degree over finite prime
  fields.
\newblock {\em Journal of Cryptology}, 18(2):79--89, 2005.

\bibitem[Deu41]{deuring}
Max Deuring.
\newblock Die {T}ypen der {M}ultiplikatorenringe elliptischer
  {F}unktionenk\"orper.
\newblock {\em Abh. Math. Sem. Hansischen Univ.}, 14:197--272, 1941.

\bibitem[Edw07]{Edwards07}
Harold~M. Edwards.
\newblock A normal form for elliptic curves.
\newblock {\em Bulletin of the American Mathematical Society}, 44(3):393--422,
  2007.

\bibitem[FR94]{fr}
Gerhard Frey and {Hans-Georg} R\"uck.
\newblock A remark concerning {$m$}-divisibility and the discrete logarithm in
  the divisor class group of curves.
\newblock {\em Math. Comp.}, 62(206):865--874, 1994.

\bibitem[Fre06]{freeman10}
David Freeman.
\newblock Constructing pairing-friendly elliptic curves with embedding degree
  10.
\newblock In {\em International Algorithmic Number Theory Symposium}, pages
  452--465. Springer, 2006.

\bibitem[FST10]{taxonomy}
David Freeman, Michael Scott, and Edlyn Teske.
\newblock A taxonomy of pairing-friendly elliptic curves.
\newblock {\em Journal of Cryptology}, 23(2):224--280, 2010.

\bibitem[Jon13]{jones}
Nathan Jones.
\newblock Elliptic aliquot cycles of fixed length.
\newblock {\em Pacific J. Math.}, 263(2):353--371, 2013.

\bibitem[Jou04]{Joux04}
Antoine Joux.
\newblock A one round protocol for tripartite {D}iffie--{H}ellman.
\newblock {\em Journal of Cryptology}, 17(4):263--276, 2004.

\bibitem[KT08]{kt}
Koray Karabina and Edlyn Teske.
\newblock On prime-order elliptic curves with embedding degrees
  k{\thinspace}={\thinspace}3, 4, and 6.
\newblock In Alfred~J. van~der Poorten and Andreas Stein, editors, {\em
  Algorithmic Number Theory}, pages 102--117, Berlin, Heidelberg, 2008.
  Springer Berlin Heidelberg.

\bibitem[Lan83]{lang_diophantine_geometry}
Serge Lang.
\newblock {\em Fundamentals of Diophantine Geometry}.
\newblock Springer New York, New York, NY, 1983.

\bibitem[{Mih}97]{preda}
Preda {Mih\v{a}ilescu}.
\newblock {\em Cyclotomy of rings \& primality testing}.
\newblock PhD thesis, {{ETH Z\"{u}rich}, Z\"{u}rich}, 1997.

\bibitem[{Mih}07]{dualelliptic}
Preda {Mih\v{a}ilescu}.
\newblock {Dual Elliptic Primes and Applications to Cyclotomy Primality
  Proving}.
\newblock {\em ArXiv e-prints}, September 2007.

\bibitem[MNT01]{mnt}
Atsuko Miyaji, Masaki Nakabayashi, and Shunzou Takano.
\newblock New explicit conditions of elliptic curve traces for {FR}-reduction.
\newblock {\em IEICE Transactions on Fundamentals of Electronics,
  Communications and Computer Sciences}, 84(5):1234--1243, 2001.

\bibitem[MOV93]{mov}
A.~J. Menezes, T.~Okamoto, and S.~A. Vanstone.
\newblock Reducing elliptic curve logarithms to logarithms in a finite field.
\newblock {\em IEEE Transactions on Information Theory}, 39(5):1639--1646, Sep
  1993.

\bibitem[Par15]{parks}
James Parks.
\newblock Amicable pairs and aliquot cycles on average.
\newblock {\em Int. J. Number Theory}, 11(6):1751--1790, 2015.

\bibitem[Par17]{parks2}
James Parks.
\newblock An asymptotic for the average number of amicable pairs for elliptic
  curves.
\newblock {\em Mathematical Proceedings of the Cambridge Philosophical
  Society}, page 1–27, 2017.

\bibitem[SA98]{satoh_araki}
Takakazu Satoh and Kiyomichi Araki.
\newblock Fermat quotients and the polynomial time discrete log algorithm for
  anomalous elliptic curves.
\newblock {\em Commentarii mathematici Universitatis Sancti Pauli},
  47(1):81--92, jun 1998.

\bibitem[Sem98]{semaev}
I.~A. Semaev.
\newblock Evaluation of discrete logarithms in a group of p-torsion points of
  an elliptic curve in characteristic p.
\newblock {\em Math. Comput.}, 67(221):353--356, January 1998.

\bibitem[Sil09]{silverman}
Joseph~H Silverman.
\newblock {\em The arithmetic of elliptic curves}, volume 106.
\newblock Springer Science \& Business Media, 2009.

\bibitem[Sma99]{smart}
N.~P. Smart.
\newblock The discrete logarithm problem on elliptic curves of trace one.
\newblock {\em Journal of Cryptology}, 12(3):193--196, Jun 1999.

\bibitem[SS11]{ss}
Joseph~H Silverman and Katherine~E Stange.
\newblock Amicable pairs and aliquot cycles for elliptic curves.
\newblock {\em Experimental Mathematics}, 20(3):329--357, 2011.

\bibitem[Sut12]{cm}
Andrew~V. Sutherland.
\newblock Accelerating the {CM} method.
\newblock {\em LMS J. Comput. Math.}, 15:172--204, 2012.

\bibitem[{The}17]{sagemath}
{The Sage Developers}.
\newblock {\em {S}ageMath, the {S}age {M}athematics {S}oftware {S}ystem
  ({V}ersion 7.5.1)}, 2017.
\newblock \url{http://www.sagemath.org}.

\bibitem[Was97]{washington}
Lawrence~C Washington.
\newblock {\em Introduction to cyclotomic fields}, volume~83.
\newblock Springer Science \& Business Media, 1997.

\end{thebibliography}
\end{document}